\documentclass[12pt]{amsart}
\usepackage[utf8]{inputenc}
\usepackage{graphicx}
\usepackage{tikz}        
\usepackage{amssymb}
\usepackage{epstopdf}
\usepackage[mathscr]{eucal}
\usepackage{amsmath,amssymb,amscd}
\usepackage{epsfig}
\usepackage{bbold}
\usepackage{hyperref}
\usepackage{MnSymbol}
\usepackage{amsfonts}
\usepackage{xcolor}%

\usepackage[numbers, sort&compress]{natbib}

\newtheorem{thm}{Theorem}
\newtheorem{lem}{Lemma}

\theoremstyle{definition}
\newtheorem{remark}{Remark}

\def \bf{\mathbf}

\title{Perturbing Masses: A Study of Centered Co-Circular Configurations in Power-Law n-Body Problems}
\begin{document}

	\maketitle
\markboth{Zhengyang Tang, Shuqiang Zhu}{  }
\vspace{-0.5cm}
\author       
\bigskip
\begin{center}
	{Zhengyang Tang$^1$, Shuqiang Zhu$^2$}\\
	{\footnotesize 
		
		School of  Mathematics,  Southwestern University of Finance and Economics, \\
		
		Chengdu 611130, China\\
		
	$^1$\texttt{kevintang\_2003@126.com},  $^2$\texttt{zhusq@swufe.edu.cn}
		
	}
	
\end{center}

\begin{abstract}
This research investigates centered co-circular central configurations in the general power-law potential $n$-body problem.  Firstly, there are no such configurations when all masses are equal, except for two; secondly,  unless all masses are equal, no such configurations exist when masses can be divided into two sets of equal masses. We adapt Wang's criterion and incorporate insights on  cyclic quadrilaterals, alongside mathematical induction. 
\end{abstract}


\textbf{2020AMS Subject Classification}: {  70F10,   70F15}.

\section{Introductions}

In the Newtonian $n$-body problem,  there is a well-known conjecture that the regular $n$-gon with equal masses is the unique co-circular central configuration whose center of mass is the center of the circle.  We consider this conjecture  in the  general power-law potential $n$-body
problem for systems  with mixed mass distributions.  Our findings reveal that, in cases where there are both  equal masses and two unequal masses, or when the masses can be divided into two groups with equal masses within each group,  no co-circular central configuration with the center of mass at the circle's center exists. This result marks a new progression towards affirming the conjecture.


The Newtonian $n$-body problem  involves characterizing the dynamic behavior of solutions to Newton's equations,  $m_k \ddot{{q}}_k =\frac{\partial U}{\partial q_k},   k=1, 2, \cdots, n$, where $U=\sum_{i< j}\frac{m_im_j}{|q_i-q_j|}$. Though initially addressed by Newton and explored by mathematicians over the centuries, the problem remains largely unsolved for $n>2$. Central configurations, just specific particle arrangements by definition, have emerged as pivotal in understanding the $n$-body problem's dynamics. They have relevance in various aspects, including homographic solutions, the analysis of collision orbits, and the bifurcation of integral manifold (cf. \cite{Wintner1941, Smale1970-2, Saari1980}).

The main focus on the topic of central configurations is the following problem:  Is the number of relative equilibria (planar central configurations) finite, in the Newtonian $n$-body problem, for any choice of positive real numbers $m_1,\cdots,m_n$
as the masses? It was
proposed by Chazy \cite{Chazy1918} and Wintner \cite{Wintner1941}, and was listed by Smale as the sixth problem  on
his list of problems for the 21-st  century \cite{Smale1998}.   Euler and Lagrange have solved this finiteness question when  $n=3$. Hampton and Moeckel \cite{Hampton2006} gave an affirmative answer for the case of $n=4$.  Albouy and  Kaloshin \cite{Albouy2012-2} gave an important partial answer to  the question   for $n=5$.  We refer the reader to the excellent review on this problem by Hampton, Moeckel \cite{Hampton2006} and Albouy, Kaloshin \cite{Albouy2012-2}.

In this paper, the co-circular central configuration whose center of mass is the center of the circle will be called \emph{centered co-circular central configuration}, following the terminology in \cite{Hampton2016}.
It is easy to see that any regular polygon with equal masses makes  a  centered co-circular central configuration, \cite{Perko1985, Wang2019}. In order to answer the question: \emph{Do there exist planar choreography solutions 
whose masses are not all equal?} Chenciner proposed another question in \cite{Chenciner2004}: \emph{Is the regular polygon with equal masses the unique centered co-circular central configuration?}  The question was also  included
in the well-known list of open problems on the classical $n$-body problems compiled by Albouy, Cabral and Santos \cite{Albouy2012-1}.  Hampton's work in \cite{Hampton2005} provided a positive answer for the case of $n = 4$. The study of $n=5$ was addressed in \cite{LLibre2015}. Wang's recent research in \cite{Wang2023} confirmed a positive  answer for $n = 5$ and $n=6$. 

This intriguing question, like many others in celestial mechanics, has also been explored in the context of the general power-law potential $n$-body problem, where the potential takes the form:
\[ U_\alpha=\sum_{i< j}\frac{m_im_j}{|q_i-q_j|^\alpha}. \]
Notably, when  $\alpha=1$, it  corresponds to the Newtonian $n$-body problem and the limiting case $\alpha=0$ corresponds to the  $n$-vortex problem. Indeed, for the limiting case of $\alpha=0$, Cors, Hall, and Roberts in \cite{Roberts2014} have established an affirmative answer to Chenciner's question for any $n$.  For $\alpha>0$, Wang's work in \cite{Wang2023} gave an positive answer for $n = 3$ and $n=4$, and furthermore, it  introduced a valuable criterion for determining the existence of the centered co-circular central configuration.


Another interesting approach to Chenciner's question was initiated by Hampton in \cite{Hampton2016}, where he proved that there are no centered co-circular central configuration formed by $n$ equal masses plus one infinitesimal mass in the case of $\alpha=1$, or, we may say that  he proved the nonexistence of such configurations  for masses in the form  of $n+\epsilon$.  This result was subsequently expanded upon  by Corbera and Valls in \cite{Corbera2019} to  general power-law potentials and masses in the form  of $n+1$, i.e., $n$ equal masses plus one arbitrary mass.   

The goal of this paper is to study the existence of centered co-circular central configuration for masses in the form of $n+1+1$ and $n+k$. More precisely, we show: 
\begin{thm}
\label{thm:twounequal} In the general power-law potential $n$-body
problem,  no centered co-circular central configurations exist where all masses are equal except for two.
\end{thm}

\begin{thm}
\label{thm:main} In the general power-law potential n-body problem,  when masses can be grouped into two sets of equal masses, no centered co-circular central configurations  exist  unless all masses are equal.
\end{thm}

Our method involves refining and extending Wang's criterion \cite{Wang2023}, along with an original result concerning cyclic quadrilaterals (see Lemma \ref{lemma:keyquadri}). Notably, our approach also incorporates the use of mathematical induction.

The paper is structured as follows. In Section \ref{sec:basic}, we briefly review the notation of centered co-circular central configurations, and list several useful lemmas.  
 In Section \ref{sec:thm1},  we prove Theorem \ref{thm:twounequal} and Theorem \ref{thm:main}.

\section{Basic settings and useful lemmas}\label{sec:basic}

Suppose that  there are n positive masses represented by ${\bf{m}}=\left(m_1, m_2, \ldots, m_n\right)$ placed around a unit circle centered at the origin in the complex plane. Their positions are given by ${\bf{q}}=\left(q_1, q_2, \ldots, q_n\right)$ in $\mathbb{C}^n$, with each position defined as $q_j=e^{\sqrt{-1} \theta_j}=\cos \theta_j+\sqrt{-1} \sin \theta_j$.
Without loss of generality, assume  that $\theta_j$ falls within the range $(0,2\pi]$, and 
\[0<\theta_1<\theta_2<\cdots<\theta_n \leq 2 \pi.  \]
We also write the  positions by  $\theta =(\theta_1, \ldots, \theta_n).$
In this way, the  mass vector   determines the order of the masses on the circle. 
Now,  the potential $U_\alpha$ is
$$
U_\alpha({\bf{m}}, \theta)=\sum_{j<k} \frac{m_j m_k}{r_{j k}^\alpha},
$$
where the distance between masses $j$ and $k$ is given by $r_{j k}$:
$$
r_{j k}=\left|2 \sin \frac{\theta_j-\theta_k}{2}\right|=\sqrt{2-2 \cos \left(\theta_j-\theta_k\right)}.
$$

It is a centered  co-circular central configuration if 
\begin{equation} \notag
\sum_{j\ne k}\frac{m_{j}(q_j-q_{k})}{r_{jk}^{\alpha+2}}+\frac{\lambda}{\alpha} q_k=0,  \ \ \  k\in\{1, \ldots, n\}.  
\end{equation}
Projecting the equations on $(-\sin \theta_k, \cos \theta_k)$ and $(\cos \theta_k, \sin \theta_k)$, \cite{Roberts2014},  an equivalent form is found as  
\begin{equation}
 \frac{\partial}{\partial\theta_{k}}U_{\alpha}=0,\ \ 
 \frac{\partial}{\partial m_{k}}U_{\alpha}=\sum_{j\ne k}\frac{m_{j}}{r_{jk}}=\frac{2\lambda}{\alpha},\ k=1,\ldots ,n.
\label{eq:cce2}
\end{equation}

The central configuration equations are invariant under the rotation. To remove the symmetry, we specify that $\theta_n=2\pi$. Let 
$\mathcal{K}_{0}=\left\{ \mathbf{\theta}: 0<\theta_{1}<\theta_{2}<\ldots<\theta_{n}=2\pi\right\} ,$
$\mathcal{CC}_{0}=\left\{ (\mathbf{{\bf m}},\theta)\  {\text satisfy}\ \eqref{eq:cce2},\theta\in\mathcal{K}_{0}\right\} .$ 

\begin{lem} [\cite{Roberts2014}]
\label{lem:Cors}For any $\mathbf{{\bf m}},$ there is a unique point
in $\mathcal{K}_{0}$ satisfying $\frac{\partial}{\partial\theta_{k}}U_{\alpha}=0,k=1,\ldots,n.$
Moreover, the critical point is a minimum, denoted by $\theta_{{\bf m}}$. 
\end{lem}


The dihedral group, $D_{n}$, acts on the set $\mathbb{R}_{+}^{n}\times\mathcal{K}_{0}$
as followes. Denote

\[
P=\left(\begin{array}{cccccc}
0 & 1 & 0 & \ldots & 0 & 0\\
0 & 0 & 1 & \ldots & 0 & 0\\
. & . & . & \ldots & . & .\\
0 & 0 & 0 & \ldots & 0 & 1\\
1 & 0 & 0 & \ldots & 0 & 0
\end{array}\right),\ \ S=\left(\begin{array}{cccccc}
0 & 0 & \ldots & 0 & 1 & 0\\
0 & 0 & \ldots & 1 & 0 & 0\\
. & . & \ldots & . & . & .\\
1 & 0 & \ldots & 0 & 0 & 0\\
0 & 0 & \ldots & 0 & 0 & 1
\end{array}\right),
\]
\[
\mathcal{P}=\left(\begin{array}{cccccc}
-1 & 1 & 0 & \ldots & 0 & 0\\
-1 & 0 & 1 & \ldots & 0 & 0\\
. & . & . & \ldots & . & .\\
-1 & 0 & 0 & \ldots & 0 & 1\\
0 & 0 & 0 & \ldots & 0 & 1
\end{array}\right),\ \ \mathcal{S}=\left(\begin{array}{cccccc}
0 & 0 & \ldots & 0 & -1 & 1\\
0 & 0 & \ldots & -1 & 0 & 1\\
. & . & \ldots & . & . & .\\
-1 & 0 & \ldots & 0 & 0 & 1\\
0 & 0 & \ldots & 0 & 0 & 1
\end{array}\right).
\]

The action of $D_{n}$ on $\mathbb{R}_{+}^{n}$ is by the matrix group
generated by $P,S$, and the action of $D_{n}$on $\mathcal{K}_{0}$
is by the matrix group generated by $\mathcal{P},\mathcal{S}$. For
any $g=P^{h}S^{l}\in D_{n},$ letting $\hat{g}=\mathcal{P}^{h}\mathcal{S}^{l},$define
the action of $D_{n}$ on $\mathbb{R}_{+}^{n}\times\mathcal{K}_{0}$
by 
\[
g\cdot(\mathbf{m,}\theta)=(g\mathbf{{\bf m}},\hat{g}\theta).
\]

\begin{lem}
Assume that $(\mathbf{m,}\theta_{\mathbf{{\bf m}}})\in\mathcal{CC}_{0}$
is a centered co-circular central configuration, then
\begin{enumerate}
\item For any $g\in D_{n}$, $g\cdot(\mathbf{m,}\theta_{\mathbf{{\bf m}}})\in\mathcal{CC}_{0}$. 
\item $U_{\alpha}(\mathbf{m,}\theta_{\mathbf{{\bf m}}})=U_{\alpha}(g\mathbf{m,}\hat{g}\theta_{\mathbf{{\bf m}}})\le U_{\alpha}(g\mathbf{m,}\theta_{\mathbf{{\bf m}}})$
and $\hat{g}\theta_{\mathbf{{\bf m}}}=\theta_{g\mathbf{{\bf m}}}.$ 
\item $\mathbf{m=}g\mathbf{{\bf m}}$ implies $\hat{g}\theta_{\mathbf{{\bf m}}}=\theta_{\mathbf{{\bf m}}}$. 
\end{enumerate}
\end{lem}

\begin{proof}
Since equations \eqref{eq:cce2} and $U_\alpha$ are invariant under the group $O(2)$ and $D_n$ is a discrete subgroup of $O(2)$, we see part (1) holds, and $U_{\alpha}(\mathbf{m,}\theta_{\mathbf{{\bf m}}})=U_{\alpha}(g\mathbf{m,}\hat{g}\theta_{\mathbf{{\bf m}}})$.  The uniqueness of the minimum implies
\[\hat{g}\theta_{\mathbf{{\bf m}}}=\theta_{g\mathbf{{\bf m}}}, \  U_{\alpha}(g\mathbf{m,}\hat{g}\theta_{\mathbf{{\bf m}}})\le U_{\alpha}(g\mathbf{m,}\theta_{\mathbf{{\bf m}}}). \]  
So part (2) is proved. 

If $\mathbf{m}=g\mathbf{m}$, then the  uniqueness of the minimum implies the equation of part (3). 
\end{proof}

Let us elaborate the second part of the above lemma. Assume that $(\mathbf{m,}\theta_{\mathbf{{\bf m}}})\in\mathcal{CC}_{0}.$
Consider the symmetric matrix $H_{\mathbf{{\bf m}}}$, which is determined
by $\mathbf{m}$ and the corresponding $\theta_{{\bf m}}$, by $\text{(\ensuremath{H_{{\bf m}})_{ij}}=1/\ensuremath{r_{ij}^{\alpha}}}$ when
$i\ne j$, and $(\ensuremath{H_{{\bf m}})_{ii}}=0.$ When considered
as a quadratic form, we can write 
\begin{equation}
U_{\alpha}(\mathbf{m},\theta_{{\bf m}})=H_{{\bf m}}(\mathbf{m})=\mathbf{m}^{T}H_{{\bf m}}\mathbf{m}.\label{equ:UandHm1}
\end{equation}
The gradient of $U_{\alpha}$ with respect to $\mathbf{m}$ is $H_{\mathbf{{\bf m}}}{\bf m}.$
Note that $(\mathbf{m},\theta_{{\bf m}})$ also satisfies $\frac{\partial}{\partial m_{k}}U_{\alpha}=\frac{2\lambda}{\alpha},k=1,\ldots,n,$
so $H_{\mathbf{{\bf m}}}\mathbf{m}=\frac{2\lambda}{\alpha}1$. Since $g\mathbf{m}-\mathbf{m}\in\mathbf{1}^{\bot}$,
where $\mathbf{1}^{\bot}=\{(x_{1},\ldots,x_{n})\in\mathbb{R}^{n},\sum x_{i}=0\},$
then 
\begin{equation}
U_{\alpha}(g\mathbf{m},\theta_{{\bf m}})=U_{\alpha}(\mathbf{m},\theta_{{\bf m}})+0+H_{{\bf m}}(g\mathbf{m}-\mathbf{m}).\label{equ:UandHm2}
\end{equation}

\begin{lem}
\label{lemma:key1} Given $\mathbf{m}$ and the corresponding $\theta_{{\bf m}}$,
if there is some $g\in D_{n}$ such that $H_{\mathbf{{\bf m}}}(g\text{\textbf{m}}-\mathbf{m})<0$,
then $(\mathbf{m},\theta_{{\bf m}})\notin\mathcal{CC}_{0}$. 
\end{lem}

\begin{remark}
	The above  two observations are essentially due to Wang \cite{Wang2023}.  While our object is the potential  $U_\alpha$, his attention is on the function in the form of $U_\alpha + \frac{U_{-2}}{K}$, where $K \ge \frac{2^{3+\alpha}}{\alpha}$. 
	\end{remark}

A co-circular configuration can be viewed as a \emph{cyclic polygon},
which is by definition a polygon with vertices upon which a circle
can be circumscribed. The following fact of cyclic quadrilateral is
important  for the study of $H_{\mathbf{{\bf m}}}(g\text{\textbf{m}}-\mathbf{m})$. 
\begin{lem}
\label{lemma:keyquadri} Cosider a cyclic quadrilateral with vertices
A, B, C, D, which are odered counterclockwise. See Figure  \ref{fig:2quadri}.  Then for
any $\alpha>0,$ it holds that

\[
\frac{1}{AC^{\alpha}}+\frac{1}{BD^{\alpha}}-(\frac{1}{AD^{\alpha}}+\frac{1}{BC^{\alpha}})<0.
\]
\end{lem}

\begin{figure}[h!]
	\centering
	\includegraphics[scale=0.7]{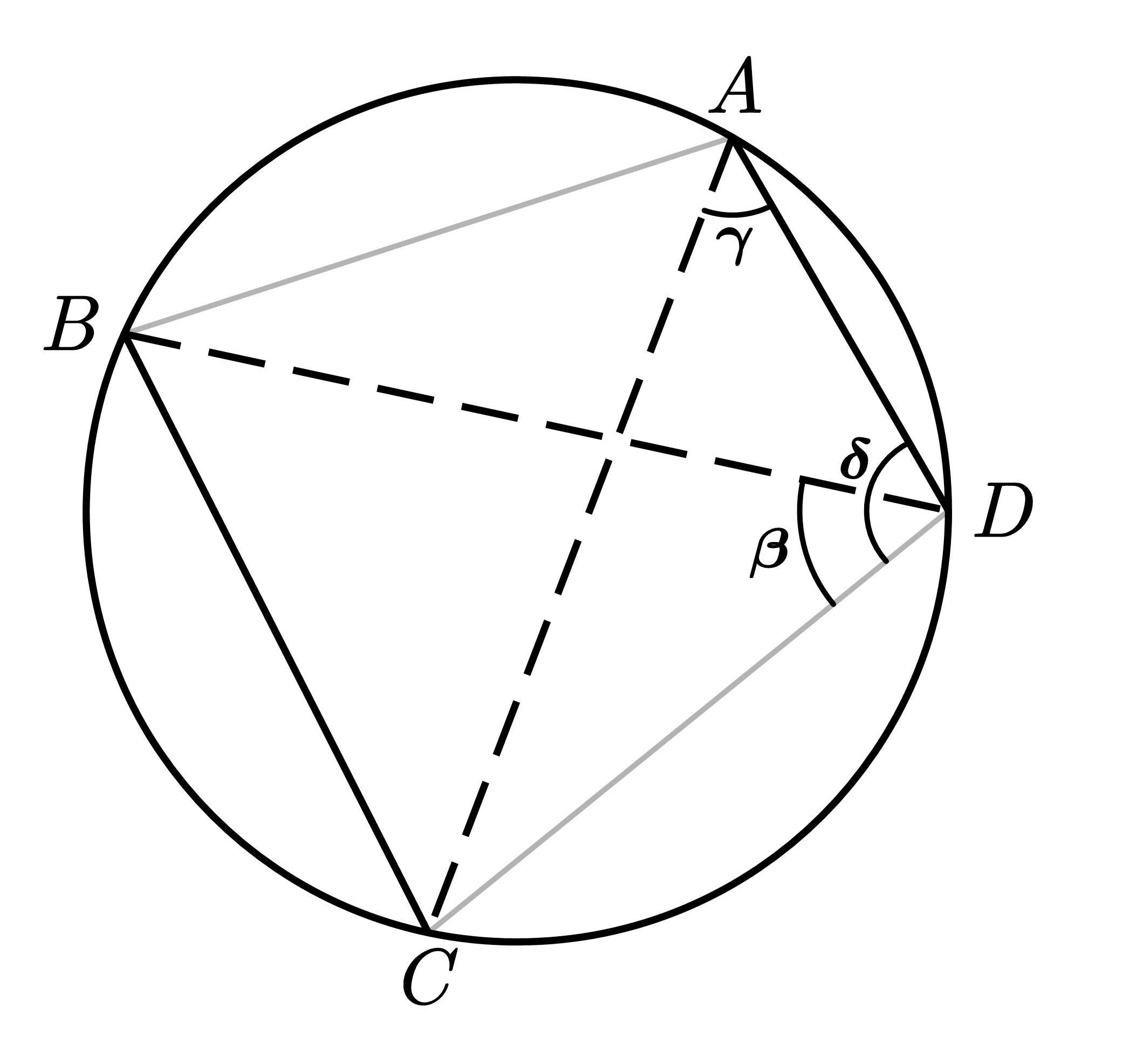}  
	\caption{One cyclic quadrilateral. The dashed lines correspond to the positive terms, while the solid black lines  correspond to the negative terms.
	}
	\label{fig:2quadri}
\end{figure}

\begin{proof}
Set $\angle CDB=\beta,\angle CDA=\delta,\angle CAD=\gamma.$ Then
$0<\beta<\delta,\beta+\gamma<\delta+\gamma<\pi.$ By the law of sines,
$\frac{AC}{\sin\delta}=\frac{BD}{\sin(\beta+\gamma)}=\frac{AD}{\sin(\pi-\delta-\gamma)}=\frac{BC}{\sin\beta}=2R,$
where $R$ is the radius of the circumcircle. Hence 
\[
\begin{array}{cc}
\frac{1}{AC^{\alpha}}+\frac{1}{BD^{\alpha}}-(\frac{1}{AD^{\alpha}}+\frac{1}{BC^{\alpha}}) & =\frac{1}{(2R)^{\alpha}}\left[\frac{1}{\sin^{\alpha}\delta}+\frac{1}{\sin^{\alpha}(\beta+\gamma)}-\frac{1}{\sin^{\alpha}(\delta+\gamma)}-\frac{1}{\sin^{\alpha}\beta}\right]\\
 & =-\frac{1}{(2R)^{\alpha}}\left[\frac{1}{sin^{\alpha}(\delta+\gamma)}-\frac{1}{sin^{\alpha}\delta}-(\frac{1}{sin^{\alpha}(\beta+\gamma)}-\frac{1}{sin^{\alpha}\beta})\right]\\
 & =-\frac{1}{\alpha(2R)^{\alpha}}\left[\int_{\delta}^{\delta+\gamma}-\frac{cos(x)}{sin^{\alpha+1}(x)}dx-\int_{\beta}^{\beta+\gamma}-\frac{cos(x)}{sin^{\alpha+1}(x)}dx\right]\\
 & =-\frac{1}{\alpha(2R)^{\alpha}}\left[\int_{0}^{\gamma}\frac{cos(x+\beta)}{sin^{\alpha+1}(x+\beta)}-\frac{cos(x+\delta)}{sin^{\alpha+1}(x+\delta)}dx\right]<0.
\end{array}
\]
In the last step, we employ the fact that $f(x)=\frac{cosx}{sin^{\alpha+1}x}$
is a decreasing function on $(0,\pi)$. Indeed, $f'(x)=-\frac{1+\alpha cos^{2}(x)}{sin^{\alpha+2}(x)}.$ 
\end{proof}

\section{Proof of the main results }\label{sec:thm1}

The main idea is to  utilize the criterion of Lemma \ref{lemma:key1}. 
If we can find some $g\in D_n$ such that the sign of $H_\mathbf{m}(g\mathbf{m}-\mathbf{m})$ is negative, then we can conclude the nonexistence of  centered co-circular central configurations. Theorem \ref{thm:twounequal} is proved in Subsection \ref{subsec:thm1}.  In Theorem \ref{thm:main}, we consider the case that  the masses consist of two group of equal masses.  We make induction on the   cardinality of the second group. 
To make the explanation more readable,  we will first show that if  the second group has cardinality 3,  there is no centered co-circular central configurations unless all masses are equal,  in Subsection \ref{subsec:thm2-1}. In the last subsection, we prove the general case of Theorem \ref{thm:main}.

\subsection{Proof of Theorem \ref{thm:twounequal}}  \label{subsec:thm1}

 In \cite{Wang2023},   Corollary 3.8, Wang has proved the result  in  the case that  the total number of particles is odd and all the masses are equal except two. 
  We only need to   discuss the case of the total number of particles being even. For completeness, we include the discussion of the case that the total number of particles is odd.  
   
\begin{proof}
\textbf{I. when $n$ is odd. }Without loss of generality, suppose
the mass vector is 
\[
\mathbf{m}=(1,\ldots,1,m_{k},1,\ldots,1,m_{n}),m_{k}\ne1,m_{n}\ne1.
\]
  Note that $S\mathbf{m}-\mathbf{m}=\pm (m_{k}-1)(0,\ldots,0,1,0,\ldots,0,-1,0,\ldots,0,0).$
Obviously, $H_{\mathbf{{\bf m}}}(S\text{\textbf{m}}-\mathbf{m})<0.$

\textbf{II. when $n$ is even. }Without loss of generality, suppose
the mass vector is  
\[
\mathbf{m}=(1,\ldots,1,m_{j},1,\ldots,1,m_{n}),m_{j}\ne1,m_{n}\ne1.
\]
There are three subcases: $1,j\ne\frac{n}{2}$; 2, $j=\frac{n}{2}$
and $m_{j}\ne m_{n}$; 3, $j=\frac{n}{2}$ and $m_{j}=m_{n}$.

\textbf{II-1.} $j\ne\frac{n}{2}$. Simliar to the case when $n$ is
odd, we have $H_{\mathbf{{\bf m}}}(S\text{\textbf{m}}-\mathbf{m})<0.$

\textbf{II-2.} $j=\frac{n}{2}$ and $m_{j}\ne m_{n}.$ Let $g=P^{\frac{n}{2}}.$
Then $P^{\frac{n}{2}}\mathbf{m}-\mathbf{m}=(m_{n}-m_{j})(0,\ldots,0,1,0,\ldots,0,-1).$
Obviously, $H_{\mathbf{{\bf m}}}(g\text{\textbf{m}}-\mathbf{m})<0$.  

\textbf{II-3.} $j=\frac{n}{2}$ and $m_{j}=m_{n}.$Let $g=P.$ Then 
\[
g\mathbf{m}-\mathbf{m}=(m_{n}-1)(0,0,\ldots,1,-1,0,0,\ldots,1,-1),\ m_{n}-1\ne0.
\]
We will show that the inequality $H_{\mathbf{{\bf m}}}(g\text{\textbf{m}}-\mathbf{m})<0$
holds again. Note that
\[
H_{{\bf m}}(g\mathbf{m}-\mathbf{m})=2(m_{n}-1)^{2}(\frac{1}{r_{\frac{n}{2}-1,n-1}^{\alpha}}+\frac{1}{r_{\frac{n}{2},n}^{\alpha}}-\frac{1}{r_{\frac{n}{2},n-1}^{\alpha}}-\frac{1}{r_{n-1,n}^{\alpha}}-\frac{1}{r_{\frac{n}{2}-1,n}^{\alpha}}-\frac{1}{r_{\frac{n}{2}-1,\frac{n}{2}}^{\alpha}}).
\]

\begin{figure}[h!]
	\centering
	\includegraphics[scale=0.7]{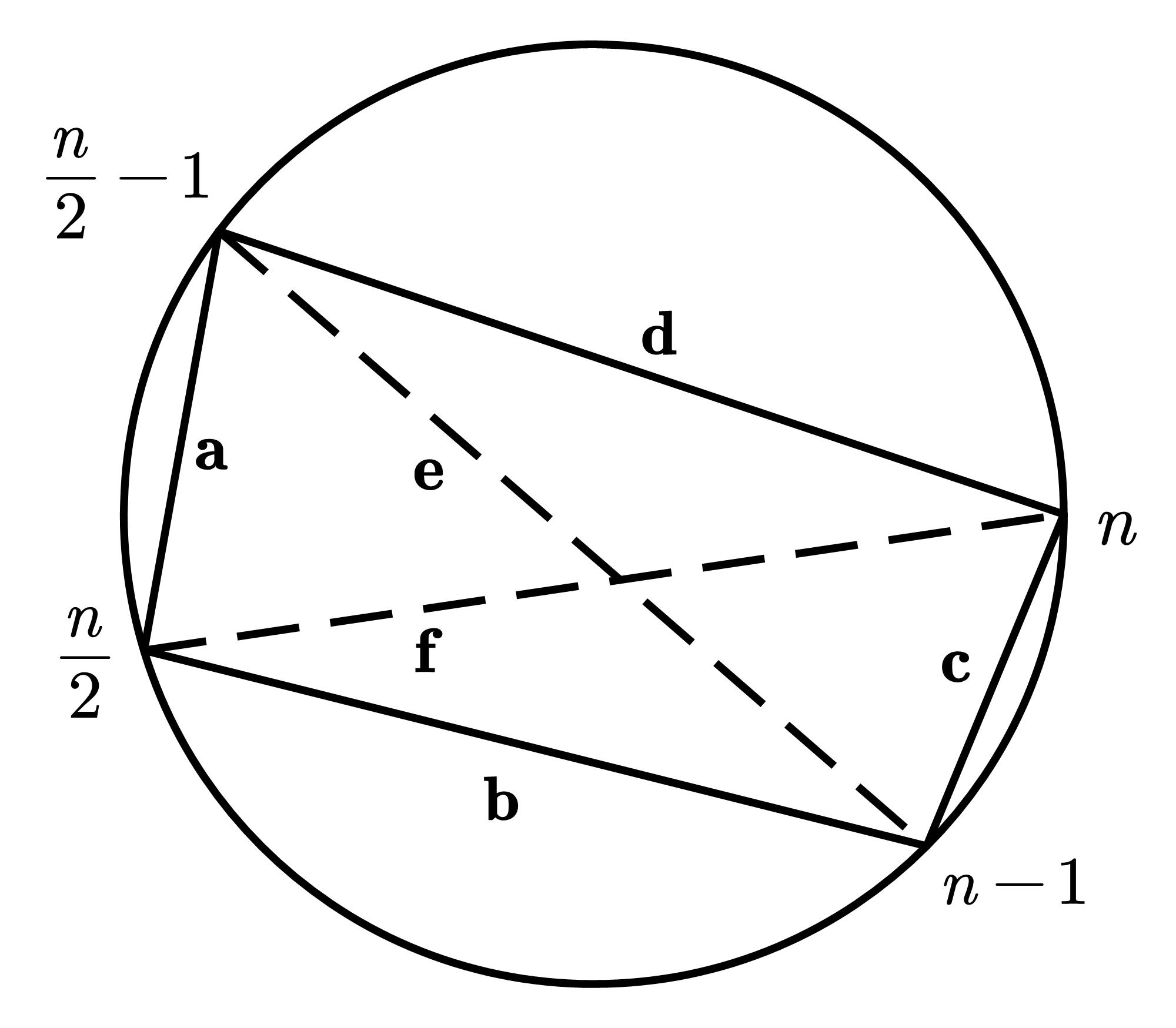}  
	\caption{The dashed lines correspond to the positive terms, while the solid  lines  correspond to the negative terms.
	}
	\label{fig:1ptolemy}
\end{figure}
Set $a=r_{\frac{n}{2}-1,\frac{n}{2}},b=r_{\frac{n}{2},n-1},c=r_{n-1,n},d=r_{n,\frac{n}{2}-1},e=r_{\frac{n}{2}-1,n-1}$
and $f=r_{\frac{n}{2},n}.$ See Figure \ref{fig:1ptolemy} . Then it suffices to show
that
\[
F=\frac{1}{e^{\alpha}}+\frac{1}{f^{\mathbf{\alpha}}}-(\frac{1}{a^{\alpha}}+\frac{1}{b^{\alpha}}+\frac{1}{c^{\alpha}}+\frac{1}{d^{\alpha}})<0.
\]
The inequality is obvious  by Lemma \ref{lemma:keyquadri}.  

When $\alpha\le1$, we would like to provide another interesting  proof.   Ptolemy's theorem says that 
$ef=ac+bd.$ Then 
\begin{align*}
F= & \frac{e^{\alpha}+f^{\alpha}}{(ef)^{\alpha}}-(\frac{a^{\alpha}+c^{\alpha}}{(ac)^{\alpha}}+\frac{b^{\alpha}+d^{\alpha}}{(bd)^{\alpha}})\\
< & \frac{a^{\alpha}+b^{\alpha}+c^{\alpha}+d^{\alpha}}{(ac+bd)^{\alpha}}-(\frac{a^{\alpha}+c^{\alpha}}{(ac)^{\alpha}}+\frac{b^{\alpha}+d^{\alpha}}{(bd)^{\alpha}})\\
< & (a^{\alpha}+c^{\alpha})[\frac{1}{(ac+bd)^{\alpha}}-\frac{1}{(ac)^{\alpha}}]+(b^{\alpha}+d^{\alpha})[\frac{1}{(ac+bd)^{\alpha}}-\frac{1}{(bd)^{\alpha}}]<0.
\end{align*}

In summary, in all cases, we have showed that there is some $g\in D_n$ such that $H_{\mathbf{{\bf m}}}(g\text{\textbf{m}}-\mathbf{m})<0$.   By Lemma \ref{lemma:key1}, there is no centered co-circular central configuration. 
\end{proof}

\subsection{A special case of Theorem \ref{thm:main}}\label{subsec:thm2-1}

The main aim of this subsection is to discuss a special case of Theorem \ref{thm:main}. For it, we first discuss a  geometric inequality of cyclic polygons.

Recall that a cyclic polygon is a polygon with vertices upon which a circle
can be circumscribed. There are many interesting researchs on the
areas of cyclic polygons. What we are interested in is one inequality  involved with the sides of an arbitrary cyclic polygon with $2n$ vertices (cyclic $2n$-gons, for short).
We will also study the inequality corresponding to some sub cyclic polygons.

 Let us fix some notations first.   We will always assume that the vertices of a cyclic $2n$-gon 
are ordered counterclockwise as $1,2,\ldots,2n.$  We refer the polygon
as $G\left\{ 1,2,\ldots,2n\right\}.$ A sub cyclic $2k$-gon consisting of 
 vertices $ i_{1,}i_{2,}\ldots,i_{2k}$, where 
 \[  \left\{ i_{1,}i_{2,}\ldots,i_{2k}\right\} \subset \{ 1,2, \ldots, 2n\}, \  i_1<i_2<\ldots i_{2k},   \] 
is refered  by $G\left\{ i_{1,}i_{2},\ldots,i_{2k}\right\} .$ For a cyclic $2n$-gon $G\left\{ 1,2,\ldots,2n\right\}, $ define 
\[  R\left(G\left\{ 1,2,\ldots,2n\right\} \right)=\sum_{p=s (mod2)}\frac{1}{r_{ps}^{\alpha}}-\sum_{p\ne s (mod2)}\frac{1}{r_{ps}^{\alpha}}. \]
Similarly, 
for a cyclic $2k$-gon $G\left\{ i_{1,}i_{2},\ldots,i_{2k}\right\} ,$ define
\[
R\left(G\left\{ i_{1,}i_{2},\ldots,i_{2k}\right\} \right)=\sum_{p=s (mod2)}\frac{1}{r_{i_{p}i_{s}}^{\alpha}}-\sum_{p\ne s (mod2)}\frac{1}{r_{i_{p}i_{s}}^{\alpha}}.
\]
For a cyclic quadrilateral $G\left\{ i_{1,}i_{2,}i_{3,}i_{4}\right\} ,$
define
\[
S\left(G\left\{ i_{1,}i_{2,}i_{3,}i_{4}\right\} \right)=\frac{1}{r_{i_{1}i_{3}}^{\alpha}}+\frac{1}{r_{i_{2}i_{4}}^{\alpha}}-\left(\frac{1}{r_{i_{2}i_{3}}^{\alpha}}+\frac{1}{r_{i_{1}i_{4}}^{\alpha}}\right).
\]

We are interested in the sign of the above defined functions. For instance, we have seen that $S\left(G\left\{ i_{1,}i_{2,}i_{3,}i_{4}\right\}\right)<0$ for any cyclic quadrilateral $G\left\{ i_{1,}i_{2,}i_{3,}i_{4}\right\}$  in Lemma \ref{lemma:keyquadri}. 
It is also clear that $R\left(G\left\{ i_{1,}i_{2,}i_{3,}i_{4}\right\}\right)<0$ by the proof of Theorem \ref{thm:twounequal}, or just by Lemma \ref{lemma:keyquadri}. It can be extended to cyclic hexagons.

\begin{lem}
\label{lemma:hexagon}For any cyclic hexagon $G\left\{ 1,2,3,4,5,6\right\} $
and any $\alpha>0,$ it holds that 
\begin{align*}
R\left(G\left\{ 1,2,3,4,5,6\right\} \right)=&\frac{1}{r_{13}^{\alpha}}+\frac{1}{r_{15}^{\alpha}}+\frac{1}{r_{24}^{\alpha}}+\frac{1}{r_{26}^{\alpha}}+\frac{1}{r_{35}^{\alpha}}+\frac{1}{r_{46}^{\alpha}}\\
&-(\frac{1}{r_{12}^{\alpha}}+\frac{1}{r_{14}^{\alpha}}+\frac{1}{r_{16}^{\alpha}}+\frac{1}{r_{23}^{\alpha}}+\frac{1}{r_{25}^{\alpha}}+\frac{1}{r_{34}^{\alpha}}+\frac{1}{r_{36}^{\alpha}}+\frac{1}{r_{45}^{\alpha}}+\frac{1}{r_{56}^{\alpha}})<0.
\end{align*}
\end{lem}

\begin{proof}
The idea is to decompose $R\left(G\left\{ 1,2,3,4,5,6\right\} \right)$ as indicated in Figure \ref{fig:3hexagon}. 

\begin{figure}[h!]
	\centering
	\includegraphics[scale=0.6]{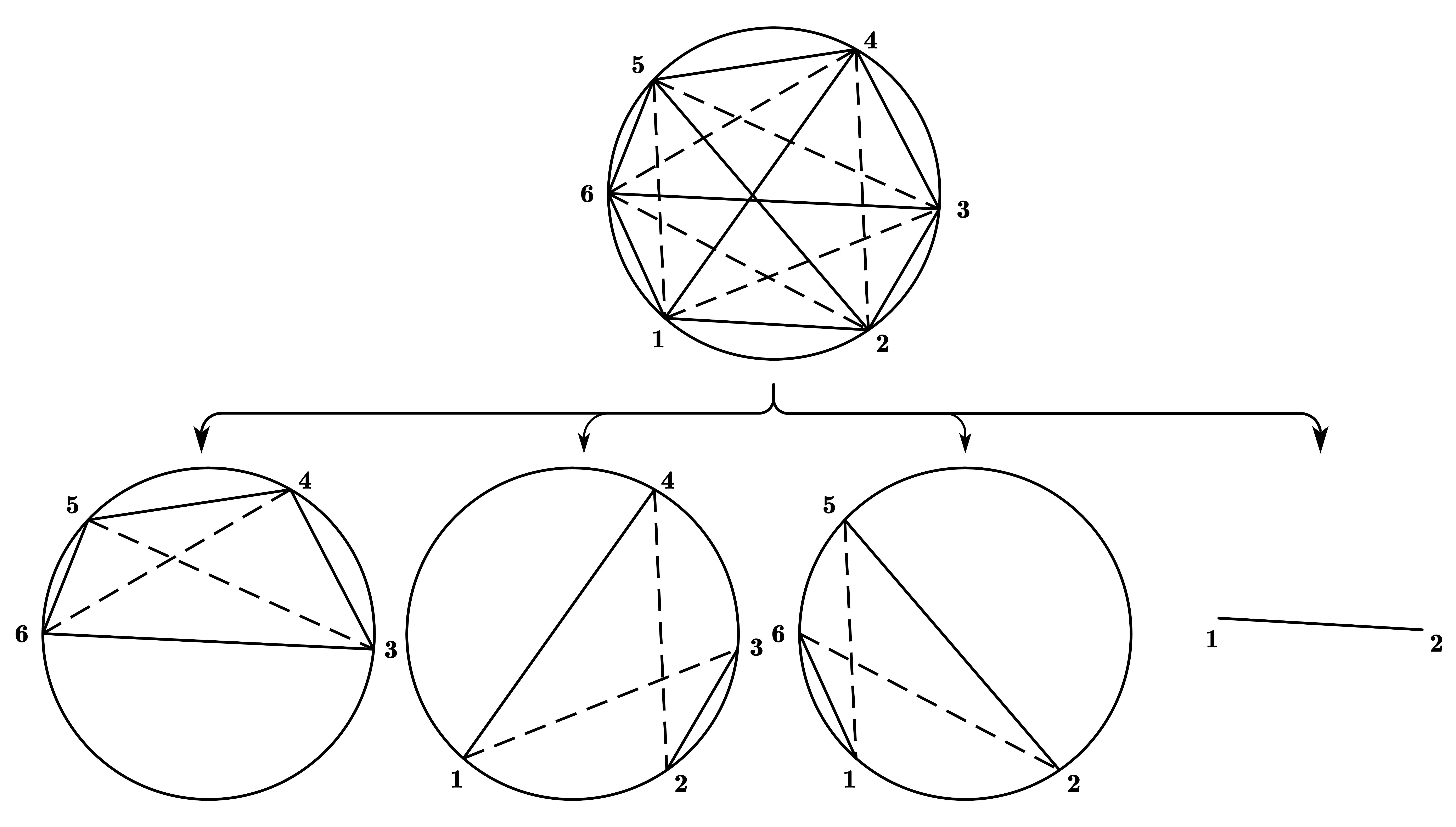}  
	\caption{The dashed lines correspond to the positive terms, while the solid  lines  correspond to the negative terms.
	}
	\label{fig:3hexagon}
\end{figure}

Note that 
\[
R\left(G\left\{ 1,2,3,4,5,6\right\} \right)=R\left(G\left\{ 3,4,5,6\right\} \right)+S\left(G\left\{ 1,2,3,4\right\} \right)+S\left(G\left\{ 1,2,5,6\right\} \right)-\frac{1}{r_{12}^{\alpha}}.
\]
Then by Lemma \ref{lemma:keyquadri}, the inequality $R\left(G\left\{ 1,2,3,4,5,6\right\} \right)<0$ holds for any cyclic hexagon.  
\end{proof}
\begin{thm}
\label{thm:three} In the general power-law potential n-body problem, assume that the masses can be divided  into two groups of equal masses, and the cardinality of the second group is 3.   There is  no centered co-circular  central configuration  unless all the masses are equal. 
\end{thm}

\begin{proof}
	Without lose of generality, assume that each mass of the first group is 1 and  each mass of the second group is $m$.
	
Firstly, we consider the case when  three $m$'s are nonadjacent.  Without losing
generality, assume the mass vector is $$\mathbf{m=}(1,\ldots,1,m,1,\ldots,1,m,1\ldots,1,m),$$
where  the three $m$'s  locate at the $i,j$ and $n$-th positions.
Since the three positions are nonadjacent, we have $1<i,i+1<j, j+1<n.$ Let $g=P.$
Then 
\[
g\mathbf{m}-\mathbf{m}=(0\ldots,0,m-1,1-m,0\ldots,0,m-1,1-m,0\ldots,0,m-1,1-m).
\]
The  $(m-1)$-terms  are at the $(i-1)$-th, the
$(j-1)$-th and the $(n-1)$-th coordinates, while the
 $(1-m)$-terms are 
at the $i$-th,  the $j$-th and  the $n$-th coordinates. Then, 
\[
H_{\mathbf{m}}(g\mathbf{m}-\mathbf{m})=2(m-1)^{2}R\left(G\left\{ i-1,i,j-1,j,n-1,n\right\} \right),
\]
where $G\left\{ i-1,i,j-1,j,n-1,n\right\} $ is the cyclic hexagon
formed by the six co-circular vertices $ i-1,i,j-1,j,n-1,n$, and  the co-circular $n$ vertices $1,2, \ldots, n$ are obtained from $\theta_{{\bf m}}$  of Lemma \ref{lem:Cors}.     By Lemma \ref{lemma:hexagon},
we have$H_{\mathbf{m}}(g\mathbf{m}-\mathbf{m})<0.$ 

Secondly, assume that two of three $m$'s are adjacent. Without losing
generality, assume the mass vector is 
$$\mathbf{m=}(1,\ldots,1,m,1,\ldots,1,m,m),$$
and the $m$'s  locate  at the $i,n-1$ and $n$-th positions.
So $1<i,i+1<n-1.$ Let $g=P.$ Then 
\[
g\mathbf{m}-\mathbf{m}=(0\ldots,0,m-1,1-m,0\ldots,0,m-1,0,1-m).
\]
Then, 
\[
H_{\mathbf{m}}(g\mathbf{m}-\mathbf{m})=2(m-1)^{2}R\left(G\left\{ i-1,i,n-2,n\right\} \right),
\]
where $G\left\{ i-1,i,n-2,n\right\} $  is  the cyclic quadrilateral 
formed by the four co-circular vertices $ i-1, i,n-2,n$, and the co-circular $n$ vertices $1,2, \ldots, n$ are obtained from $\theta_{{\bf m}}$  of Lemma \ref{lem:Cors}.  
By Lemma \ref{lemma:keyquadri}, we have $H_{\mathbf{m}}(g\mathbf{m}-\mathbf{m})<0.$

Thirdly, assume that the three $m$'s are adjacent. Without losing
generality, assume the mass vector is 
$$\mathbf{m=}(1,\ldots,1,m,m,m).$$
Let $g=P.$ Then 
\[
g\mathbf{m}-\mathbf{m}=(0\ldots,0,m-1,0,0,1-m).
\]
Then obviously, $H_{\mathbf{m}}(g\mathbf{m}-\mathbf{m})<0.$

In summary, in all cases, we have showed that there is some $g\in D_n$ such that $H_{\mathbf{{\bf m}}}(g\text{\textbf{m}}-\mathbf{m})<0$.   By Lemma \ref{lemma:key1}, there is no centered co-circular central configuration. 
\end{proof}

\subsection{The general case of Theorem \ref{thm:main}}

We first generalize Lemma \ref{lemma:hexagon} to all cyclic $2n$-gons,
via induction. 
\begin{lem}
\label{lem:decomposition}For any cylic $2n$-gon $G\left\{ 1,2,\ldots,2n\right\} $, it always holds that 
\[  R\left(G\left\{ 1,2,\ldots,2n\right\} \right)=\sum_{p=s (mod2)}\frac{1}{r_{ps}^{\alpha}}-\sum_{p\ne s (mod2)}\frac{1}{r_{ps}^{\alpha}}<0. \]
\end{lem}

\begin{proof}
We proceed by induction on the number $n$. First, note that it is true for $n=2$ by Lemma \ref{lemma:keyquadri}, for $n=3$ by Lemma \ref{lem:decomposition}.
Now assume that it holds for number less than $n$. For the case of  $n$, 
decompose $R\left(G\left\{ 1,2,\ldots, 2n\right\} \right)$ as indicated in Figure \ref{fig:42n-gon}. 

\begin{figure}[h!]
	\centering
	\includegraphics[scale=0.6]{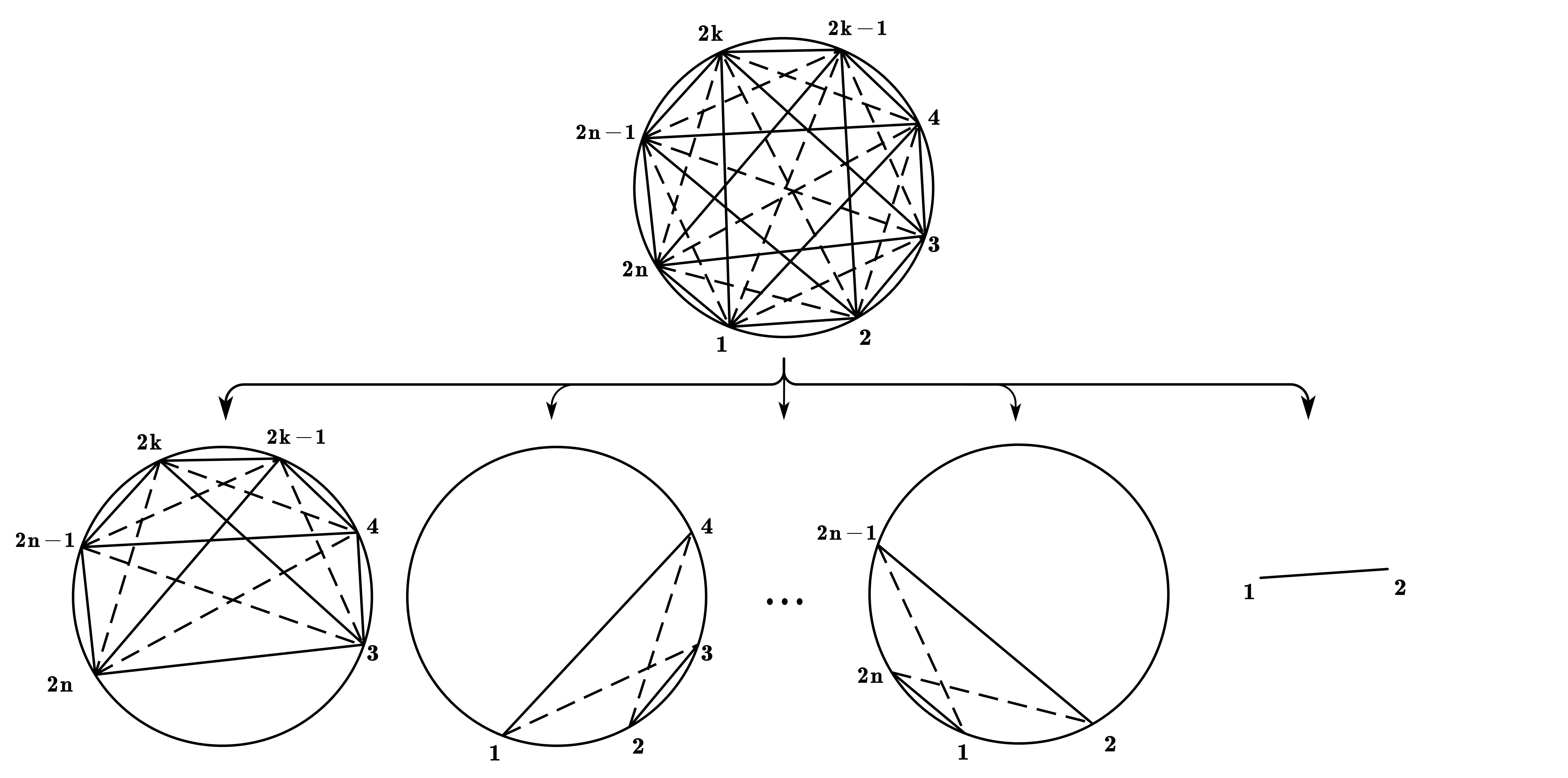}  
	\caption{The dashed lines correspond to the positive terms, while the solid  lines  correspond to the negative terms.
	}
	\label{fig:42n-gon}
\end{figure}

Note that 
\begin{align*}
R\left(G\left\{ 1,2,\ldots,2n\right\} \right)=&R\left(G\left\{ 3,4,\ldots,2n\right\} \right)+S\left(G\left\{ 1,2,3,4\right\} \right)\\
&+S\left(G\left\{ 1,2,5,6\right\} \right)+\ldots+S\left(G\left\{ 1,2,2n-1,2n\right\} \right)-\frac{1}{r_{12}^{\alpha}}.
\end{align*}
Then by Lemma \ref{lemma:keyquadri} and the hypotheses,   the inequality $R\left(G\left\{ 1,2,\ldots, 2n\right\} \right)<0$ holds for any cyclic 2n-gon.   

In summary, the inequality $R\left(G\left\{ 1,2,\ldots, 2n\right\} \right)<0$ holds for any cyclic 2n-gon.   
\end{proof}
\begin{proof}
{[proof of Theorem \ref{thm:main}]} 

	Without lose of generality, assume that each mass of the first group is 1 and  each mass of the second group is $m$. Suppose that the cardinality  is $n$ for  the first group and is $k$ for the second, and  $k\le n.$
	
Firstly, we consider the case when  $k$ $m$'s are nonadjacent.  Without losing
generality, assume that  the $k$$m$'s  locate  the $i_{1},\ldots,i_{k}$-th
positions and $i_k=n$. Then $1<i_{1}, i_{s}+1<i_{s+1}$ for $1\le s\le k-1$ and $i_k=n$. 
Similar to the proof of Theorem \ref{thm:three}, for $g=P$, the
vector $g\mathbf{m}-\mathbf{m}$ has $n-k$ zeros. Neglecting those
zeros, and dividing it by $m-1,$ the vector $g\mathbf{m}-\mathbf{m}$
consists of $k$ $1$'s, and $k$ $-1$'s.  The $1$'s and $-1$'s appear consectively. Then 
\[
H_{\mathbf{m}}(g\mathbf{m}-\mathbf{m})=2(m-1)^{2}R\left(G\left\{ i_{1}-1,i_{1},i_{2}-1,i_{2},\ldots,i_{k}-1,i_{k}\right\} \right),
\]
where $G\left\{ i_{1}-1,i_{1},i_{2}-1,i_{2},\ldots,i_{k}-1,i_{k}\right\}$
is the  cyclic $2k$-gon formed by the $2k$ co-circular vertices $i_{1}-1,i_{1},i_{2}-1,i_{2},\ldots,i_{k}-1,i_{k}$, 
and  the co-circular $n+k$ vertices $1,2, \ldots, n+k$ are obtained from $\theta_{{\bf m}}$  of Lemma \ref{lem:Cors}.     By Lemma \ref{lem:decomposition}, 
we have$H_{\mathbf{m}}(g\mathbf{m}-\mathbf{m})<0.$

Secondly, assume that some of $k$ $m$'s are adjacent.  Similar to the
proof of Theorem \ref{thm:three}, it would always hold that $H_{\mathbf{m}}(Pm-\mathbf{m})<0$ and we would like not to repeat the proof.

In summary, in all cases, we have showed that there is some $g\in D_n$ such that $H_{\mathbf{{\bf m}}}(g\text{\textbf{m}}-\mathbf{m})<0$.   By Lemma \ref{lemma:key1}, there is no centered co-circular central configuration. 
\end{proof}

\section*{Acknowledgment}
The first author  would like to thank the Innovation Training Program for College Students of Southwestern University of Finance and Economics for its support in all aspects of the project. He was able to complete the  work smoothly because of the school's platform and organization.

\end{document}